\documentclass[10pt]{amsart}
\usepackage[a4paper,margin=22mm]{geometry}
\usepackage{amsmath,amssymb,amsthm,mathtools,booktabs}
\usepackage[protrusion=true,expansion=false]{microtype}
\usepackage[hidelinks]{hyperref}
\hypersetup{pdftitle={An improved upper bound for recovering pairs},pdfauthor={Simone Costa}}
\newtheorem{theorem}{Theorem}[section]
\newtheorem{lemma}[theorem]{Lemma}
\newtheorem{proposition}[theorem]{Proposition}
\theoremstyle{remark}
\newtheorem{remark}[theorem]{Remark}
\theoremstyle{plain}

\newcommand{\cA}{\mathcal A}
\newcommand{\cB}{\mathcal B}
\newcommand{\cC}{\mathcal C}
\newcommand{\cD}{\mathcal D}
\newcommand{\Prob}{\mathbb P}
\newcommand{\E}{\mathbb E}

\title{An improved upper bound for recovering pairs}
\author{Simone Costa}
\thanks{DICATAM, Universit\`a degli Studi di Brescia, Via Branze 43, 25123 Brescia, Italy. Email: \href{mailto:simone.costa@unibs.it}{simone.costa@unibs.it}.}
\date{}

\begin{document}

\begin{abstract}
We prove that every recovering pair $(\cA,\cB)$ of nonempty families of subsets of an $n$-element set satisfies
\[
 |\cA||\cB|\le\left(\frac{223949}{100000}\right)^n=(2.23949)^n.
\]
A result of Mond, Souza and Versteegen, combined with the sharp $(9/4)^n$ bound of Fang and Huang for cancellative pairs, gives $(2.2499)^n$ for recovering pairs. We improve this to $(2.23949)^n$. The proof combines two entropy inequalities. One is derived directly from the recovering property, while the other applies the Fang--Huang bound to families obtained from the original pair by a construction of Mond, Souza and Versteegen. The two inequalities are then combined by a convex combination.
\end{abstract}
\maketitle

\section{Introduction}\label{sec:intro}
A pair $(\cA,\cB)$ of nonempty families of subsets of $[n]=\{1,\ldots,n\}$ is \emph{recovering} if, for all $A,A'\in\cA$ and $B,B'\in\cB$,
\begin{equation}\label{eq:recovery}
 A\setminus B=A'\setminus B'\Longrightarrow A=A',\qquad
 B\setminus A=B'\setminus A'\Longrightarrow B=B'.
\end{equation}
The sandglass conjecture of Ahlswede and Simonyi~\cite{AS} asserts that
$|\cA||\cB|\le 2^n$. Equality is attained by taking the power sets of two complementary parts of $[n]$.

A pair $(\cA,\cB)$ is \emph{cancellative} if, for all $A,A'\in\cA$ and $B,B'\in\cB$,
\begin{equation}\label{eq:cancellative}
 A\setminus B=A'\setminus B\Longrightarrow A=A',\qquad
 B\setminus A=B'\setminus A\Longrightarrow B=B'.
\end{equation}
Thus every recovering pair is cancellative. Holzman and K\"orner~\cite{HK} proved an upper bound of approximately $2.3264^n$ for cancellative pairs. Solt\'esz~\cite{Soltesz} obtained $2.284^n$ for recovering pairs; Janzer~\cite{Janzer} obtained $2.2682^n$ for cancellative pairs; and Nair and Yazdanpanah~\cite{NY} obtained $2.2663^n$ for recovering pairs. Mond, Souza and Versteegen~\cite{MSV} subsequently proved the recovering-pair bound $2.2543^n$. Recovering pairs are also equivalent to zero-error code pairs for the deterministic AND--OR interference channel; see~\cite{NY}.

The distinction between recovering and cancellative pairs is essential for the present argument. Tolhuizen~\cite{Tolhuizen} constructed cancellative pairs with product $(2.25-o(1))^n$. Fang and Huang~\cite{FH} have proved the matching finite upper bound $(9/4)^n$ for every cancellative pair. Combined with the more general estimate of Mond, Souza and Versteegen~\cite[Theorem~1.3]{MSV}, this gives the recovering-pair bound $2.2499^n$. Chang~\cite{Chang} subsequently gave an independent hypercontractive proof of the cancellative-pair bound. Hence a universal bound below $(9/4)^n$ for recovering pairs must use information not available from cancellativity alone.

\begin{theorem}\label{thm:main}
Let $(\cA,\cB)$ be a recovering pair of nonempty families of subsets of an $n$-element set. Then
\[
 |\cA||\cB|\le\left(\frac{223949}{100000}\right)^n=(2.23949)^n.
\]
\end{theorem}

The proof combines two upper bounds for $\log_2(|\cA||\cB|)$. The first is derived directly from the recovering property. The second applies Theorem~\ref{thm:FH} to the families in Proposition~\ref{prop:MSV}. We take a convex combination of the two bounds and choose its weights and parameters jointly, following the same general principle as in~\cite[Section~5]{DFCDhashing}.

Section~\ref{sec:prelim} collects notation and the entropy identities used in the proofs. Section~\ref{sec:bounds} proves the bounds, and Section~\ref{sec:proof} combines them to prove Theorem~\ref{thm:main}. Appendix~\ref{app:computer} reduces Proposition~\ref{prop:pointwise} and Proposition~\ref{prop:obstruction} to finite rational checks. The computational supplement describes these checks, while the ancillary files provide the exact verifiers, certificates, and outputs. Both are available at Zenodo, \href{https://doi.org/10.5281/zenodo.23093260}{doi:10.5281/zenodo.23093260}.

\section{Notation and entropy facts}\label{sec:prelim}
This section fixes the notation and states the entropy identities used repeatedly in Sections~\ref{sec:bounds} and~\ref{sec:proof}. All random variables below take values in finite sets, and all logarithms are base $2$. For a random variable $X$ with mass function $p_X$ we write
\[
 H(X)=-\sum_x p_X(x)\log_2 p_X(x),
 \qquad 0\log_2 0:=0.
\]
For random variables $X,Y$ we write
\[
 H(X\mid Y)=\sum_y\Prob(Y=y)H(X\mid Y=y),
 \qquad
 I(X;Y)=H(X)-H(X\mid Y).
\]
We use $\Prob(E)$ for the probability of an event $E$ and $\E[f(X)]=\sum_x f(x)p_X(x)$ for expectation. The entropy identities used below are
\begin{align}
 H(X,Y)&=H(X)+H(Y\mid X),\label{eq:chain-two}\\
 H(X_1,\ldots,X_m)&=\sum_{k=1}^m H(X_k\mid X_1,\ldots,X_{k-1}),\notag\\
 H(X\mid Y,Z)&\le H(X\mid Y),\label{eq:conditioning}\\
 H(X\mid Y)&\ge0,\notag\\
 I(X;Y)&\ge0,\notag\\
 H(X_1,\ldots,X_m)&\le\sum_{k=1}^m H(X_k).\label{eq:subadditivity}
\end{align}
In particular, fix an order $\pi$ of $[n]$ and set
\[
 c_{\pi(k)}^X=H(X_{\pi(k)}\mid X_{\pi(1)},\ldots,X_{\pi(k-1)})
 \qquad(1\le k\le n).
\]
The chain rule gives $\sum_{i=1}^n c_i^X=H(X)$, and \eqref{eq:conditioning} gives $c_i^X\le H(X_i)$. For a subset $S\subseteq[n]$, write its elements in the order induced by $\pi$ as $\pi(k_1),\ldots,\pi(k_r)$ with $k_1<\cdots<k_r$. Then
\[
 H(X_S)=\sum_{j=1}^r H\bigl(X_{\pi(k_j)}\mid X_{\pi(k_1)},\ldots,X_{\pi(k_{j-1})}\bigr)
 \ge \sum_{j=1}^r c_{\pi(k_j)}^X,
\]
because $c_{\pi(k_j)}^X$ is defined by conditioning on all coordinates preceding $\pi(k_j)$ in the order $\pi$. Thus
\begin{equation}\label{eq:projection-chain}
 \sum_{i=1}^n c_i^X=H(X),
 \qquad c_i^X\le H(X_i),
 \qquad H(X_S)\ge\sum_{i\in S}c_i^X
 \quad(S\subseteq[n]).
\end{equation}
If $S$ is an initial segment of the order $\pi$, the two conditioning sets in the chain-rule expansion above coincide, so the last inequality in \eqref{eq:projection-chain} is an equality.

For $t\in[0,1]$, let
\[
 h(t)=-t\log_2 t-(1-t)\log_2(1-t),\qquad h(0)=h(1)=0.
\]
Thus a Bernoulli random variable of mean $t$ has entropy $h(t)$. The function $h$ is continuous on $[0,1]$ and concave. We also write $v_+=\max\{v,0\}$.

If $X$ is a random subset of $[n]$, write $X_i=\mathbf 1_{\{i\in X\}}$. For $J\subseteq[n]$, put $J^c=[n]\setminus J$ and write $X_J=X\cap J$ and $X_{J^c}=X\cap J^c$; when entropy is applied, these sets are identified with their indicator vectors. The total correlation of $X_J$ is
\begin{equation}\label{eq:TC}
 \operatorname{TC}(X_J)=\sum_{i\in J}H(X_i)-H(X_J),\qquad
 \operatorname{TC}(X_\varnothing)=0.
\end{equation}
It is nonnegative by \eqref{eq:subadditivity}. We abbreviate $\operatorname{TC}(X)=\operatorname{TC}(X_{[n]})$. For every $J\subseteq[n]$, expanding \eqref{eq:TC} and using \eqref{eq:chain-two} gives
\begin{equation}\label{eq:TC-split}
 \operatorname{TC}(X)=\operatorname{TC}(X_J)+\operatorname{TC}(X_{J^c})+I(X_J;X_{J^c}).
\end{equation}

For a recovering pair $(\cA,\cB)$, unless stated otherwise, $A$ and $B$ denote independent random sets uniformly distributed on $\cA$ and $\cB$, respectively. We use the notation
\begin{equation}\label{eq:basic-notation}
 T=|\cA||\cB|,\qquad
 a_i=\Prob(i\in A),\qquad b_i=\Prob(i\in B),\qquad
 p_i=a_i(1-b_i),\qquad q_i=b_i(1-a_i),
\end{equation}
and write $Z=A\setminus B$ and $W=B\setminus A$. Uniformity and independence give
\[
 H(A)=\log_2|\cA|,\qquad H(B)=\log_2|\cB|,\qquad
 H(A,B)=H(A)+H(B)=\log_2T.
\]
Moreover, $\Prob(Z_i=1)=p_i$ and $\Prob(W_i=1)=q_i$, so
\[
 H(Z_i)=h(p_i),\qquad H(W_i)=h(q_i).
\]

\section{Two estimates for recovering pairs}\label{sec:bounds}
Subsection~\ref{subsec:recovering-entropy} derives an entropy bound from the recovering property. Subsection~\ref{subsec:cancellative-entropy} derives a second bound from Proposition~\ref{prop:MSV} and Theorem~\ref{thm:FH}.

\subsection{An entropy estimate}\label{subsec:recovering-entropy}
For $x,y\in[0,1]$, put $p=x(1-y)$ and $q=y(1-x)$, and define
\begin{equation}\label{eq:g}
 g(x,y)=h\bigl(x(1-y)\bigr)-(1-y)h(x).
\end{equation}
If $X$ and $Y$ are independent Bernoulli random variables with respective means $x$ and $y$, then $X(1-Y)$ is Bernoulli with mean $x(1-y)$, while
\[
 H(X(1-Y)\mid Y)=(1-y)H(X)= (1-y)h(x).
\]
Consequently
\[
 I(X(1-Y);Y)=h\bigl(x(1-y)\bigr)-(1-y)h(x)=g(x,y)\ge0.
\]

\begin{lemma}\label{lem:projection}
Let $A$ and $B$ be independent random subsets of $[n]$, set
$a_i=\Prob(i\in A)$, $b_i=\Prob(i\in B)$, and $Z=A\setminus B$.
Then, for every $J\subseteq[n]$,
\begin{equation}\label{eq:projection}
 \operatorname{TC}(Z)+H(B_J)\ge\sum_{i\in J}g(a_i,b_i).
\end{equation}
Moreover, let $t_1,\ldots,t_n\in[0,1]$, choose a permutation $\pi$ of $[n]$ such that
$t_{\pi(1)}\ge\cdots\ge t_{\pi(n)}$, and define
\[
 c^B_{\pi(k)}=H\bigl(B_{\pi(k)}\mid B_{\pi(1)},\ldots,B_{\pi(k-1)}\bigr)
 \qquad(1\le k\le n).
\]
Then
\begin{equation}\label{eq:weighted-proj}
 \operatorname{TC}(Z)+\sum_{i=1}^n t_i c_i^B
 \ge\sum_{i=1}^n t_i g(a_i,b_i).
\end{equation}
For $n=0$, both assertions are understood as $0\ge0$.
\end{lemma}

\begin{proof}
By \eqref{eq:TC-split},
\[
 \operatorname{TC}(Z)=\operatorname{TC}(Z_J)+\operatorname{TC}(Z_{J^c})+I(Z_J;Z_{J^c})
 \ge \operatorname{TC}(Z_J).
\]
For every $U\subseteq J$, on the event $B_J=U$ we have $Z_J=A_{J\setminus U}$. Since $A$ and $B$ are independent,
\begin{align*}
 H(Z_J\mid B_J)
 &=\sum_{U\subseteq J}\Prob(B_J=U)H(A_{J\setminus U})\\
 &\le\sum_{U\subseteq J}\Prob(B_J=U)
       \sum_{i\in J\setminus U}h(a_i)\\
 &=\sum_{i\in J}\Prob(i\notin B)h(a_i)
 =\sum_{i\in J}(1-b_i)h(a_i),
\end{align*}
where the inequality is \eqref{eq:subadditivity}. Moreover,
\[
 H(Z_J)\le H(Z_J,B_J)=H(B_J)+H(Z_J\mid B_J)
\]
by \eqref{eq:chain-two}. Hence
\[
 H(Z_J)\le H(B_J)+\sum_{i\in J}(1-b_i)h(a_i).
\]
Using $H(Z_i)=h(a_i(1-b_i))$ and the definition of $g$, we obtain
\begin{align*}
 \operatorname{TC}(Z_J)
 &=\sum_{i\in J}h(a_i(1-b_i))-H(Z_J)\\
 &\ge\sum_{i\in J}\bigl[h(a_i(1-b_i))-(1-b_i)h(a_i)\bigr]-H(B_J)\\
 &=\sum_{i\in J}g(a_i,b_i)-H(B_J),
\end{align*}
which proves \eqref{eq:projection}.

We now prove \eqref{eq:weighted-proj}. Let $t_{\max}=\max_i t_i$, with $t_{\max}=0$ when $n=0$, and for $0\le u<t_{\max}$ put
\[
 J_u=\{i:t_i>u\}.
\]
The definition is chosen so that coordinate $i$ belongs to $J_u$ for a set of values of $u$ of length $t_i$. Because the weights are ordered decreasingly by $\pi$, each $J_u$ is an initial segment of that order, and \eqref{eq:projection-chain} gives
$H(B_{J_u})=\sum_{i\in J_u}c_i^B$. We can therefore integrate \eqref{eq:projection} over $u\in[0,t_{\max}]$. Since
$\mathbf 1_{\{i\in J_u\}}=\mathbf 1_{\{u<t_i\}}$, the left-hand side becomes
\begin{align*}
 &\int_0^{t_{\max}}\left[\operatorname{TC}(Z)+H(B_{J_u})\right]du\\
 &\qquad=t_{\max}\operatorname{TC}(Z)
   +\sum_i c_i^B\int_0^{t_{\max}}\mathbf 1_{\{u<t_i\}}du
 =t_{\max}\operatorname{TC}(Z)+\sum_i t_i c_i^B,
\end{align*}
and similarly
\[
 \int_0^{t_{\max}}\sum_{i\in J_u}g(a_i,b_i)\,du
 =\sum_i t_i g(a_i,b_i).
\]
Therefore
\[
 t_{\max}\operatorname{TC}(Z)+\sum_i t_i c_i^B
 \ge\sum_i t_i g(a_i,b_i).
\]
Since $0\le t_{\max}\le1$ and $\operatorname{TC}(Z)\ge0$, replacing
$t_{\max}\operatorname{TC}(Z)$ by the larger quantity $\operatorname{TC}(Z)$ proves \eqref{eq:weighted-proj}.
\end{proof}

For $\lambda\ge0$ define
\begin{equation}\label{eq:Vhat}
\begin{split}
 (1+\lambda)\widehat V_\lambda(x,y)={}&h(p)+h(q)
 + (\lambda-x)_+h(y)+(\lambda-y)_+h(x)\\
 &-(x-\lambda)_+g(x,y)-(y-\lambda)_+g(y,x),
\end{split}
\end{equation}

\begin{lemma}\label{lem:strongV}
Let $(\cA,\cB)$ be a recovering pair of nonempty families on $[n]$, and let $A,B,T,a_i,b_i$ be as in \eqref{eq:basic-notation}. Then, for every $\lambda\ge0$,
\begin{equation}\label{eq:strongV}
 \log_2T\le\sum_{i=1}^n\widehat V_\lambda(a_i,b_i).
\end{equation}
\end{lemma}

\begin{proof}
The recovering property implies that $A$ is determined by $Z=A\setminus B$ and $B$ is determined by $W=B\setminus A$. Hence $H(A,Z)=H(Z)$ and, by \eqref{eq:chain-two},
\begin{align*}
 H(Z)
 &=H(A)+H(Z\mid A)\\
 &=H(A)+\sum_{S\subseteq[n]}\Prob(A=S)H(S\setminus B\mid A=S)\\
 &=H(A)+\sum_{S\subseteq[n]}\Prob(A=S)H(B_S).
\end{align*}
The last equality uses independence of $A$ and $B$ and the bijection $B_S\mapsto S\setminus B$. Interchanging $A$ and $B$ gives
\begin{equation}\label{eq:recoveryentropy}
 H(Z)=H(A)+\sum_{S\subseteq[n]}\Prob(A=S)H(B_S),\qquad
 H(W)=H(B)+\sum_{S\subseteq[n]}\Prob(B=S)H(A_S).
\end{equation}

Choose an order of the coordinates by decreasing
$t_i=(a_i-\lambda)_+$ and let $c_i^B$ be the conditional entropies defined from that order as in \eqref{eq:projection-chain}. For each fixed set $S$, \eqref{eq:projection-chain} gives
$H(B_S)\ge\sum_{i\in S}c_i^B$. Multiplying by $\Prob(A=S)$ and summing over $S$ gives
\[
 \sum_{S\subseteq[n]}\Prob(A=S)H(B_S)\ge\sum_i a_i c_i^B.
\]
Since $a_i=\lambda+t_i-(\lambda-a_i)_+$,
\begin{align*}
 \sum_i a_i c_i^B
 &=\lambda\sum_i c_i^B+\sum_i t_i c_i^B
   -\sum_i(\lambda-a_i)_+c_i^B\\
 &=\lambda H(B)+\sum_i t_i c_i^B
   -\sum_i(\lambda-a_i)_+c_i^B.
\end{align*}
By \eqref{eq:projection-chain}, $c_i^B\le H(B_i)=h(b_i)$. Since each coefficient $(\lambda-a_i)_+$ is nonnegative,
\[
 -\sum_i(\lambda-a_i)_+c_i^B
 \ge -\sum_i(\lambda-a_i)_+h(b_i).
\]
Combining the last three displays gives
\begin{equation}\label{eq:projection-B}
 \sum_{S\subseteq[n]}\Prob(A=S)H(B_S)
 \ge \lambda H(B)+\sum_i t_i c_i^B-\sum_i(\lambda-a_i)_+h(b_i).
\end{equation}
Applying the same argument with $A$ and $B$ interchanged, with
$s_i=(b_i-\lambda)_+$, gives
\begin{equation}\label{eq:projection-A}
 \sum_{S\subseteq[n]}\Prob(B=S)H(A_S)
 \ge \lambda H(A)+\sum_i s_i c_i^A-\sum_i(\lambda-b_i)_+h(a_i),
\end{equation}
where $c_i^A$ denotes the conditional entropies associated with the order chosen for the weights $s_i$.

By \eqref{eq:TC} and the definitions of $p_i$ and $q_i$,
\[
 H(Z)=\sum_i h(p_i)-\operatorname{TC}(Z),
 \qquad
 H(W)=\sum_i h(q_i)-\operatorname{TC}(W).
\]
Combining the first identity in \eqref{eq:recoveryentropy} with \eqref{eq:projection-B} therefore gives
\begin{align*}
 H(A)+\lambda H(B)
 \le{}&\sum_i h(p_i)-\operatorname{TC}(Z)-\sum_i t_i c_i^B\\
 &+\sum_i(\lambda-a_i)_+h(b_i).
\end{align*}
Similarly, the second identity in \eqref{eq:recoveryentropy} and \eqref{eq:projection-A} give
\begin{align*}
 H(B)+\lambda H(A)
 \le{}&\sum_i h(q_i)-\operatorname{TC}(W)-\sum_i s_i c_i^A\\
 &+\sum_i(\lambda-b_i)_+h(a_i).
\end{align*}
Adding these two inequalities and using
$H(A)+H(B)=\log_2T$ yields
\begin{align*}
 (1+\lambda)\log_2T\le{}&\sum_i\bigl[h(p_i)+h(q_i)
 +(\lambda-a_i)_+h(b_i)+(\lambda-b_i)_+h(a_i)\bigr]\\
 &-\left[\operatorname{TC}(Z)+\sum_i t_i c_i^B\right]
 -\left[\operatorname{TC}(W)+\sum_i s_i c_i^A\right].
\end{align*}
By \eqref{eq:weighted-proj},
\[
 \operatorname{TC}(Z)+\sum_i t_i c_i^B
 \ge\sum_i t_i g(a_i,b_i),
\]
and the same inequality with $A$ and $B$ interchanged gives
\[
 \operatorname{TC}(W)+\sum_i s_i c_i^A
 \ge\sum_i s_i g(b_i,a_i).
\]
Since $t_i=(a_i-\lambda)_+$ and $s_i=(b_i-\lambda)_+$, we obtain
\begin{align*}
 (1+\lambda)\log_2T
 \le \sum_i\Bigl[&
 h(p_i)+h(q_i)+(\lambda-a_i)_+h(b_i)+(\lambda-b_i)_+h(a_i)\\
 &-(a_i-\lambda)_+g(a_i,b_i)
 -(b_i-\lambda)_+g(b_i,a_i)\Bigr].
\end{align*}
The bracket is
$(1+\lambda)\widehat V_\lambda(a_i,b_i)$ by \eqref{eq:Vhat}. Dividing by
$1+\lambda>0$ proves \eqref{eq:strongV}.
\end{proof}

\begin{lemma}\label{lem:remainder}
Let $A$ and $B$ be independent random subsets of $[n]$, set
$a_i=\Prob(i\in A)$, $b_i=\Prob(i\in B)$, and $Z=A\setminus B$.
For every $J\subseteq[n]$,
\begin{align}\label{eq:remainder}
 &\operatorname{TC}(Z)+H(B_J)-\sum_{i\in J}g(a_i,b_i)\notag\\
 &\quad=\operatorname{TC}(Z)-\operatorname{TC}(Z_J)
 +H(B_J\mid Z_J)
 +\sum_{U\subseteq J}\Prob(B_J=U)\operatorname{TC}(A_{J\setminus U}).
\end{align}
Every term on the second line is nonnegative. Moreover, for weights $t_i\in[0,1]$ ordered as in Lemma~\ref{lem:projection}, with $t_{\max}=\max_i t_i$ and $J_u=\{i:t_i>u\}$,
\begin{align}\label{eq:weighted-remainder}
 &\operatorname{TC}(Z)+\sum_i t_i c_i^B-\sum_i t_i g(a_i,b_i)\notag\\
 &\quad=(1-t_{\max})\operatorname{TC}(Z)
 +\int_0^{t_{\max}}\Bigl[
 \operatorname{TC}(Z)-\operatorname{TC}(Z_{J_u})
 +H(B_{J_u}\mid Z_{J_u})\notag\\
 &\hspace{27mm}+\sum_{U\subseteq J_u}\Prob(B_{J_u}=U)
 \operatorname{TC}(A_{J_u\setminus U})
 \Bigr]du.
\end{align}
For $n=0$, the identity is interpreted with $t_{\max}=0$.
\end{lemma}

\begin{proof}
For a fixed $U\subseteq J$, conditioning on $B_J=U$ makes $Z_J$ the projection of $A$ onto $J\setminus U$. Independence therefore gives
\begin{subequations}\label{eq:conditional-TC-identities}
\begin{equation}\label{eq:conditional-projection}
 H(Z_J\mid B_J)
 =\sum_{U\subseteq J}\Prob(B_J=U)H(A_{J\setminus U}).
\end{equation}
Also
\begin{equation}\label{eq:marginal-projection}
 \sum_{i\in J}(1-b_i)h(a_i)
 =\sum_{U\subseteq J}\Prob(B_J=U)
   \sum_{i\in J\setminus U}h(a_i).
\end{equation}
Subtracting \eqref{eq:conditional-projection} from \eqref{eq:marginal-projection} gives
\begin{equation}\label{eq:conditional-TC}
 \sum_{i\in J}(1-b_i)h(a_i)-H(Z_J\mid B_J)
 =\sum_{U\subseteq J}\Prob(B_J=U)\operatorname{TC}(A_{J\setminus U}).
\end{equation}
\end{subequations}
Now expand the left-hand side of \eqref{eq:remainder}. Since
$H(Z_i)=h(a_i(1-b_i))$,
\begin{align*}
 &\operatorname{TC}(Z)+H(B_J)-\sum_{i\in J}g(a_i,b_i)\\
 &\quad=\operatorname{TC}(Z)+H(B_J)
 -\sum_{i\in J}H(Z_i)+\sum_{i\in J}(1-b_i)h(a_i)\\
 &\quad=\operatorname{TC}(Z)-\operatorname{TC}(Z_J)
 +H(B_J)-H(Z_J)+\sum_{i\in J}(1-b_i)h(a_i).
\end{align*}
The chain rule gives
\[
 H(B_J)-H(Z_J)=H(B_J\mid Z_J)-H(Z_J\mid B_J).
\]
Substituting this identity and then \eqref{eq:conditional-TC} gives \eqref{eq:remainder}. Finally,
\[
 \operatorname{TC}(Z)-\operatorname{TC}(Z_J)
 =\operatorname{TC}(Z_{J^c})+I(Z_J;Z_{J^c})\ge0
\]
by \eqref{eq:TC-split}; the other terms on the right-hand side of \eqref{eq:remainder} are nonnegative by definition or by \eqref{eq:subadditivity}.
Integrate \eqref{eq:remainder} for $J=J_u$ over $0\le u\le t_{\max}$. Since $J_u$ is an initial segment of the chosen order, \eqref{eq:projection-chain} gives
$H(B_{J_u})=\sum_{i\in J_u}c_i^B$. Therefore
\begin{align*}
 \int_0^{t_{\max}}H(B_{J_u})\,du&=\sum_i t_i c_i^B,\\
 \int_0^{t_{\max}}\sum_{i\in J_u}g(a_i,b_i)\,du
 &=\sum_i t_i g(a_i,b_i).
\end{align*}
The integral of $\operatorname{TC}(Z)$ is $t_{\max}\operatorname{TC}(Z)$. Thus the integrated form of \eqref{eq:remainder} has
$t_{\max}\operatorname{TC}(Z)+\sum_i t_i c_i^B-\sum_i t_i g(a_i,b_i)$ on the left. Adding
$(1-t_{\max})\operatorname{TC}(Z)$ to both sides gives \eqref{eq:weighted-remainder}.
\end{proof}

\begin{lemma}\label{lem:conditional}
Let $(\cA,\cB)$ be a recovering pair of nonempty families on $[n]$, and use the notation in \eqref{eq:basic-notation}. Put
\[
 \Sigma_Z=\sum_{i=1}^n h(p_i),\qquad
 \Sigma_W=\sum_{i=1}^n h(q_i).
\]
Then, for every $J\subseteq[n]$,
\begin{align}\label{eq:conditional}
 H(B_J\mid Z_J)&\ge
 \left[\log_2T-\Sigma_Z-\sum_{i\notin J}(1-a_i)h(b_i)\right]_+,\notag\\
 H(A_J\mid W_J)&\ge
 \left[\log_2T-\Sigma_W-\sum_{i\notin J}(1-b_i)h(a_i)\right]_+.
\end{align}
\end{lemma}

\begin{proof}
Because $A$ is determined by $Z$ and $Z$ is a function of $(A,B)$,
\[
 H(B\mid Z)=H(A,B\mid Z)=H(A,B)-H(Z)=\log_2T-H(Z).
\]
Because $A$ is a function of $Z$,
\[
 H(B_{J^c}\mid Z)=H(B_{J^c}\mid A,Z).
\]
Fix a value $A=S$. For $i\in S$, the relation $Z=A\setminus B$ gives
$B_i=1-Z_i$, so these coordinates of $B$ are determined once $A=S$ and $Z$ are known. Consequently, once $A=S$ and $Z$ are fixed, the vector $B_{J^c}$ is determined by its coordinates in $J^c\setminus S$. Thus
\[
 H(B_{J^c}\mid A=S,Z)
 =H(B_{J^c\setminus S}\mid A=S,Z)
 \le H(B_{J^c\setminus S}\mid A=S).
\]
The last inequality is \eqref{eq:conditioning}. Since $A$ and $B$ are independent,
$H(B_{J^c\setminus S}\mid A=S)=H(B_{J^c\setminus S})$, and \eqref{eq:subadditivity} gives
\[
 H(B_{J^c\setminus S})
 \le\sum_{i\in J^c\setminus S}h(b_i).
\]
Averaging over $A$ therefore yields
\begin{align*}
 H(B_{J^c}\mid Z)
 &\le \sum_{S\subseteq[n]}\Prob(A=S)
      \sum_{i\in J^c\setminus S}h(b_i)\\
 &=\sum_{i\notin J}\Prob(i\notin A)h(b_i)
 =\sum_{i\notin J}(1-a_i)h(b_i).
\end{align*}

Next, the chain rule and \eqref{eq:conditioning} give
\begin{align*}
 H(B\mid Z)
 &=H(B_J\mid Z)+H(B_{J^c}\mid B_J,Z)\\
 &\le H(B_J\mid Z)+H(B_{J^c}\mid Z).
\end{align*}
Hence
\[
 H(B_J\mid Z)\ge H(B\mid Z)-H(B_{J^c}\mid Z).
\]
Since $Z$ contains $Z_J$, conditioning on $Z$ can only decrease entropy, so
\[
 H(B_J\mid Z_J)\ge H(B_J\mid Z).
\]
Combining these inequalities with the bounds for $H(B\mid Z)$ and $H(B_{J^c}\mid Z)$ obtained above gives
\begin{align*}
 H(B_J\mid Z_J)
 &\ge \log_2T-H(Z)-\sum_{i\notin J}(1-a_i)h(b_i)\\
 &\ge \log_2T-\Sigma_Z-\sum_{i\notin J}(1-a_i)h(b_i).
\end{align*}
Let
\[
 R_J=\log_2T-\Sigma_Z-\sum_{i\notin J}(1-a_i)h(b_i).
\]
We have proved $H(B_J\mid Z_J)\ge R_J$. Since conditional entropy is also nonnegative,
\[
 H(B_J\mid Z_J)\ge\max\{R_J,0\}=[R_J]_+,
\]
which is the first inequality in \eqref{eq:conditional}. Interchanging $A$ with $B$ and $Z$ with $W$ gives the second inequality.
\end{proof}

\begin{lemma}\label{lem:integrated}
Let $(\cA,\cB)$ be a recovering pair of nonempty families on $[n]$, and use the notation in \eqref{eq:basic-notation} and the quantities $\Sigma_Z,\Sigma_W$ defined in Lemma~\ref{lem:conditional}. Fix $\lambda\ge0$ and put
\[
 t_i=(a_i-\lambda)_+,
 \qquad s_i=(b_i-\lambda)_+,
 \qquad J_u=\{i:t_i>u\},
 \qquad K_u=\{i:s_i>u\}.
\]
Let $t_{\max}=\max_i t_i$ and $s_{\max}=\max_i s_i$, both interpreted as $0$ when $n=0$. Then
\begin{align}\label{eq:integrated}
 &(1+\lambda)\left[\sum_i\widehat V_\lambda(a_i,b_i)-\log_2T\right]\notag\\
 &\quad\ge
 \int_0^{t_{\max}}
 \left[\log_2T-\Sigma_Z-\sum_{i\notin J_u}(1-a_i)h(b_i)\right]_+du\notag\\
 &\qquad+
 \int_0^{s_{\max}}
 \left[\log_2T-\Sigma_W-\sum_{i\notin K_u}(1-b_i)h(a_i)\right]_+du.
\end{align}
\end{lemma}

\begin{proof}
The proof of Lemma~\ref{lem:strongV} gives
\begin{align}
 (1+\lambda)\log_2T\le{}&\sum_i\bigl[h(p_i)+h(q_i)
 +(\lambda-a_i)_+h(b_i)+(\lambda-b_i)_+h(a_i)\bigr]\notag\\
 &-\left[\operatorname{TC}(Z)+\sum_i t_i c_i^B\right]
 -\left[\operatorname{TC}(W)+\sum_i s_i c_i^A\right].\label{eq:integrated-upper}
\end{align}
On the other hand, summing \eqref{eq:Vhat} over the coordinates gives
\begin{align}
 (1+\lambda)\sum_i\widehat V_\lambda(a_i,b_i)
 =\sum_i\bigl[&
 h(p_i)+h(q_i)+(\lambda-a_i)_+h(b_i)+(\lambda-b_i)_+h(a_i)\notag\\
 &-t_i g(a_i,b_i)-s_i g(b_i,a_i)\bigr].\label{eq:integrated-Vsum}
\end{align}
Subtracting \eqref{eq:integrated-upper} from \eqref{eq:integrated-Vsum} yields
\begin{align*}
 &(1+\lambda)\left[\sum_i\widehat V_\lambda(a_i,b_i)-\log_2T\right]\\
 &\quad\ge
 \left[\operatorname{TC}(Z)+\sum_i t_i c_i^B-\sum_i t_i g(a_i,b_i)\right]\\
 &\qquad+
 \left[\operatorname{TC}(W)+\sum_i s_i c_i^A-\sum_i s_i g(b_i,a_i)\right].
\end{align*}

Apply \eqref{eq:weighted-remainder} to the first bracket. Its right-hand side is
$(1-t_{\max})\operatorname{TC}(Z)$ plus an integral whose summands are all nonnegative. In particular,
\[
 \operatorname{TC}(Z)+\sum_i t_i c_i^B-\sum_i t_i g(a_i,b_i)
 \ge\int_0^{t_{\max}}H(B_{J_u}\mid Z_{J_u})\,du.
\]
Applying the same identity with $A$ and $B$ interchanged gives
\[
 \operatorname{TC}(W)+\sum_i s_i c_i^A-\sum_i s_i g(b_i,a_i)
 \ge\int_0^{s_{\max}}H(A_{K_u}\mid W_{K_u})\,du.
\]
For every $0\le u\le t_{\max}$, Lemma~\ref{lem:conditional} with $J=J_u$ gives
\[
 H(B_{J_u}\mid Z_{J_u})\ge
 \left[\log_2T-\Sigma_Z-\sum_{i\notin J_u}(1-a_i)h(b_i)\right]_+.
\]
The same lemma with $A,B$ interchanged and $J=K_u$ gives the corresponding bound for $H(A_{K_u}\mid W_{K_u})$. Substituting these two pointwise bounds into the two integrals proves \eqref{eq:integrated}.
\end{proof}

For $0\le\lambda\le\tau\le1$ define
\begin{equation}\label{eq:W}
\begin{aligned}
 (1+2\tau-\lambda)W_{\lambda,\tau}(x,y)
 ={}&(1+\lambda)\widehat V_\lambda(x,y)
 +(\tau-\lambda)[h(p)+h(q)]\\
 &+(1-x)h(y)(\tau-\max\{\lambda,x\})_+\\
 &+(1-y)h(x)(\tau-\max\{\lambda,y\})_+.
\end{aligned}
\end{equation}

\begin{theorem}\label{thm:W}
Let $(\cA,\cB)$ be a recovering pair of nonempty families on $[n]$, and let $A,B,T,a_i,b_i$ be as in \eqref{eq:basic-notation}. Then, for every $0\le\lambda\le\tau\le1$,
\begin{equation}\label{eq:Wglobal}
 \log_2T\le\sum_{i=1}^n W_{\lambda,\tau}(a_i,b_i).
\end{equation}
\end{theorem}

\begin{proof}
Let $\Sigma_Z,\Sigma_W$ be as in Lemma~\ref{lem:conditional}, and set
\[
 t_i=(a_i-\lambda)_+,\qquad s_i=(b_i-\lambda)_+,\qquad
 t_{\max}=\max_i t_i,\qquad s_{\max}=\max_i s_i,
\]
with both maxima interpreted as $0$ when $n=0$. For $u\ge0$ put
$J_u=\{i:t_i>u\}$ and $K_u=\{i:s_i>u\}$. Finally, for $v\in[\lambda,1]$ define
\[
 C_B(v)=\sum_{i:a_i\le v}(1-a_i)h(b_i),\qquad
 C_A(v)=\sum_{i:b_i\le v}(1-b_i)h(a_i).
\]
For $u\ge0$,
\[
 i\notin J_u
 \Longleftrightarrow (a_i-\lambda)_+\le u
 \Longleftrightarrow a_i\le\lambda+u,
\]
and the same equivalence holds with $b_i$ and $K_u$ in place of $a_i$ and $J_u$. Hence, whenever $0\le u\le1-\lambda$,
\begin{equation}\label{eq:CA-CB-identities}
 \sum_{i\notin J_u}(1-a_i)h(b_i)=C_B(\lambda+u),
 \qquad
 \sum_{i\notin K_u}(1-b_i)h(a_i)=C_A(\lambda+u).
\end{equation}
Lemma~\ref{lem:conditional} with $J=\varnothing$ gives
\begin{equation}\label{eq:empty-set}
 \log_2T-\Sigma_Z-C_B(1)\le0,
 \qquad
 \log_2T-\Sigma_W-C_A(1)\le0.
\end{equation}
If $t_{\max}\le u\le1-\lambda$, then $J_u=\varnothing$ and $\lambda+u\ge\max_i a_i$, so $C_B(\lambda+u)=C_B(1)$. Using \eqref{eq:empty-set}, the first integrand in \eqref{eq:integrated} is therefore zero. Similarly, the second integrand is zero for $s_{\max}\le u\le1-\lambda$. In formulas,
\begin{equation}\label{eq:integrand-zero}
 \begin{aligned}
 \left[\log_2T-\Sigma_Z-C_B(\lambda+u)\right]_+&=0
 &&(t_{\max}\le u\le1-\lambda),\\
 \left[\log_2T-\Sigma_W-C_A(\lambda+u)\right]_+&=0
 &&(s_{\max}\le u\le1-\lambda).
 \end{aligned}
\end{equation}
Because $\tau-\lambda\le1-\lambda$, equations \eqref{eq:CA-CB-identities} and \eqref{eq:integrand-zero} justify replacing the upper limit of either integral in \eqref{eq:integrated} by $\tau-\lambda$: if the new interval is longer, \eqref{eq:integrand-zero} shows that the added part is zero; if it is shorter, only nonnegative terms are removed. We then use
$[r]_+\ge r$ and substitute $v=\lambda+u$. This gives
\begin{align*}
 &(1+\lambda)\left[\sum_i\widehat V_\lambda(a_i,b_i)-\log_2T\right]\\
 &\quad\ge
 \int_\lambda^\tau[\log_2T-\Sigma_Z-C_B(v)]\,dv
 +\int_\lambda^\tau[\log_2T-\Sigma_W-C_A(v)]\,dv.
\end{align*}
Now
\begin{align*}
 \int_\lambda^\tau C_B(v)\,dv
 &=\sum_i(1-a_i)h(b_i)
   \int_\lambda^\tau \mathbf 1_{\{a_i\le v\}}\,dv\\
 &=\sum_i(1-a_i)h(b_i)(\tau-\max\{\lambda,a_i\})_+,
\end{align*}
and, in the same way,
\[
 \int_\lambda^\tau C_A(v)\,dv
 =\sum_i(1-b_i)h(a_i)(\tau-\max\{\lambda,b_i\})_+.
\]
Consequently,
\begin{align*}
 &(1+\lambda)\left[\sum_i\widehat V_\lambda(a_i,b_i)-\log_2T\right]\\
 &\quad\ge
 2(\tau-\lambda)\log_2T-(\tau-\lambda)(\Sigma_Z+\Sigma_W)\\
 &\qquad-
 \sum_i(1-a_i)h(b_i)(\tau-\max\{\lambda,a_i\})_+\\
 &\qquad-
 \sum_i(1-b_i)h(a_i)(\tau-\max\{\lambda,b_i\})_+.
\end{align*}
Move the terms on the right other than $2(\tau-\lambda)\log_2T$ to the opposite side and use
\[
 \Sigma_Z+\Sigma_W=\sum_i\bigl[h(p_i)+h(q_i)\bigr].
\]
We obtain
\begin{align*}
 &(1+2\tau-\lambda)\log_2T\\
 &\quad\le\sum_i\Bigl[
 (1+\lambda)\widehat V_\lambda(a_i,b_i)
 +(\tau-\lambda)\bigl(h(p_i)+h(q_i)\bigr)\\
 &\hspace{35mm}
 +(1-a_i)h(b_i)(\tau-\max\{\lambda,a_i\})_+\\
 &\hspace{35mm}
 +(1-b_i)h(a_i)(\tau-\max\{\lambda,b_i\})_+
 \Bigr].
\end{align*}
Each bracket is
$(1+2\tau-\lambda)W_{\lambda,\tau}(a_i,b_i)$ by \eqref{eq:W}. Since
$1+2\tau-\lambda>0$, division proves \eqref{eq:Wglobal}. When $\tau=\lambda$, the added terms vanish and \eqref{eq:Wglobal} reduces to Lemma~\ref{lem:strongV}. The case $n=0$ is $0\le0$, and the endpoint values of the marginals are included because $h$ is continuous on $[0,1]$.
\end{proof}

\subsection{A second estimate using cancellative pairs}\label{subsec:cancellative-entropy}
We shall use the Fang--Huang bound for cancellative pairs below. They use the equivalent union formulation: a pair $(\cC,\cD)$ is cancellative if, for all $C,C'\in\cC$ and $D,D'\in\cD$,
\[
 C\cup D=C'\cup D\Longrightarrow C=C',\qquad
 C\cup D=C\cup D'\Longrightarrow D=D'.
\]
Indeed, for fixed sets $C,C',D$,
\[
 C\cup D=C'\cup D
 \quad\Longleftrightarrow\quad
 C\setminus D=C'\setminus D,
\]
and the same equivalence with the two families interchanged gives the second condition. Thus this formulation is equivalent to \eqref{eq:cancellative}.

We use the following consequences of Mond, Souza and Versteegen~\cite[Observations~3.2--3.3 and Proposition~3.4]{MSV}.

\begin{proposition}\label{prop:MSV}
Let $(\cA,\cB)$ be a recovering pair of nonempty families on $[n]$. Let $A$ be uniform on $\cA$ and set $a_i=\Prob(i\in A)$. Fix $C\in\cA$ and $B_0\in\cB$, put $S=C\setminus B_0$, and for $U\subseteq S$ define
\[
 \phi_{C,S,U}:\{A'\in\cA:A'\cap S=U\}\longrightarrow 2^{[n]\setminus C},
 \qquad \phi_{C,S,U}(A')=A'\setminus C,
\]
\[
 \psi_{C,S}:\{B'\in\cB:C\setminus B'=S\}\longrightarrow 2^{[n]\setminus C},
 \qquad \psi_{C,S}(B')=B'\setminus C.
\]
Let
\begin{equation}\label{eq:filters}
 \cA_{C,S,U}=\phi_{C,S,U}(\{A'\in\cA:A'\cap S=U\}),\qquad
 \cB_{C,S}=\psi_{C,S}(\{B'\in\cB:C\setminus B'=S\}).
\end{equation}
Then $\phi_{C,S,U}$ and $\psi_{C,S}$ are injective for every $U\subseteq S$. Moreover, there exists $U\subseteq S$ such that $\cA_{C,S,U}$ and $\cB_{C,S}$ are nonempty, form a cancellative pair, and
\begin{equation}\label{eq:largestatom}
 |\cA_{C,S,U}|\ge |\cA|\,2^{-\sum_{i\in S}h(a_i)}.
\end{equation}
\end{proposition}

\begin{proof}
Every recovering pair is cancellative by \eqref{eq:recovery} and \eqref{eq:cancellative}. Moreover, the set $S$ in the proposition has the required form $S=C\setminus B_0$ by construction. Thus \cite[Observation~3.2]{MSV} applies to the two maps defined above and gives the stated injectivity. With the same $C$ and $S$, \cite[Observation~3.3]{MSV} shows that whenever the two image families are nonempty they form a cancellative pair.

For the uniform random set $A$ and the marginal probabilities $a_i=\Prob(i\in A)$, \cite[Proposition~3.4]{MSV} provides a choice of $U\subseteq S$ satisfying \eqref{eq:largestatom}. The family $\cB_{C,S}$ is nonempty because $B_0$ belongs to the domain of $\psi_{C,S}$ and $\psi_{C,S}(B_0)=B_0\setminus C$. The right-hand side of \eqref{eq:largestatom} is positive, so the corresponding family $\cA_{C,S,U}$ is nonempty as well. Observation~3.3 of \cite{MSV} therefore gives cancellativity of this pair.
\end{proof}

We use the following theorem of Fang and Huang. For $m=0$, the same inequality holds because both nonempty families are then $\{\varnothing\}$.

\begin{theorem}\label{thm:FH}
Let $(\cC,\cD)$ be a cancellative pair, in the sense of \eqref{eq:cancellative}, of nonempty families of subsets of an $m$-element set, where $m\ge0$. Then
\begin{equation}\label{eq:FH}
 |\cC||\cD|\le(9/4)^m.
\end{equation}
\end{theorem}

Put $L=\log_2(9/4)$ and, for $x,y\in[0,1]$, define
\begin{equation}\label{eq:F}
 F(x,y)=x(1-y)h(x)+y(1-x)h(y)
       +h\bigl(x(1-y)\bigr)+h\bigl(y(1-x)\bigr).
\end{equation}

\begin{lemma}\label{lem:direct}
Let $(\cA,\cB)$ be a recovering pair of nonempty families on $[n]$, and let $A,B,T,a_i,b_i$ be as in \eqref{eq:basic-notation}. Then
\begin{equation}\label{eq:direct}
 3\log_2T\le\sum_{i=1}^n\bigl[F(a_i,b_i)+L(2-a_i-b_i)\bigr].
\end{equation}
\end{lemma}

\begin{proof}
Let $Z=A\setminus B$, and fix $S\subseteq[n]$ with $\Prob(Z=S)>0$. By the recovering property, $Z=S$ determines a unique set $C\in\cA$ with $A=C$. Since the event has positive probability, there is $B_0\in\cB$ such that $S=C\setminus B_0$. Choose $U\subseteq S$ as in Proposition~\ref{prop:MSV}.

Because $Z=S$ forces $A=C$, independence of $A$ and $B$ gives
\begin{align*}
 \Prob(Z=S)
 &=\Prob(A=C)\Prob(C\setminus B=S)\\
 &=\frac{1}{|\cA|}
   \frac{|\{B'\in\cB:C\setminus B'=S\}|}{|\cB|}.
\end{align*}
By Proposition~\ref{prop:MSV}, $\psi_{C,S}$ is injective, so the cardinality in the numerator is $|\cB_{C,S}|$. Thus
\begin{equation}\label{eq:atomprob}
 \Prob(Z=S)=\frac{|\cB_{C,S}|}{T}.
\end{equation}
The families $\cA_{C,S,U}$ and $\cB_{C,S}$, with $U$ chosen as in Proposition~\ref{prop:MSV}, are nonempty and form a cancellative pair on the $n-|C|$ coordinates of $[n]\setminus C$. Hence Theorem~\ref{thm:FH} gives
\begin{equation}\label{eq:FHfilter}
 |\cA_{C,S,U}|\,|\cB_{C,S}|\le(9/4)^{n-|C|}.
\end{equation}
By \eqref{eq:largestatom},
\[
 \frac1{|\cA_{C,S,U}|}
 \le \frac{2^{\sum_{i\in S}h(a_i)}}{|\cA|}.
\]
Therefore \eqref{eq:FHfilter} implies
\[
 |\cB_{C,S}|
 \le \frac{(9/4)^{n-|C|}}{|\cA|}
      2^{\sum_{i\in S}h(a_i)}.
\]
Using \eqref{eq:atomprob}, we obtain
\[
 \Prob(Z=S)
 \le \frac{(9/4)^{n-|C|}}{T|\cA|}
      2^{\sum_{i\in S}h(a_i)}.
\]
Taking negative base-$2$ logarithms gives
\[
 -\log_2\Prob(Z=S)
 \ge \log_2T+\log_2|\cA|-(n-|C|)L-\sum_{i\in S}h(a_i).
\]

We now average this inequality over $S=Z$. Since the recovering property determines $A$ from $Z$, the set denoted by $C$ for the realized value $S=Z$ is the random set $A$. Hence
\[
 \E|C|=\E|A|=\sum_i a_i.
\]
Moreover,
\[
 \E\!\left[\sum_{i\in Z}h(a_i)\right]
 =\sum_i \Prob(i\in Z)h(a_i)
 =\sum_i p_i h(a_i),
\]
because independence gives
$\Prob(i\in Z)=\Prob(i\in A,i\notin B)=a_i(1-b_i)=p_i$.
Finally,
\[
 \sum_{\substack{S\subseteq[n]\\ \Prob(Z=S)>0}}
 \Prob(Z=S)\bigl[-\log_2\Prob(Z=S)\bigr]=H(Z).
\]
Therefore
\begin{align*}
 H(Z)
 \ge \log_2T+\log_2|\cA|-
     \left(n-\sum_i a_i\right)L-
     \sum_i p_i h(a_i).
\end{align*}
Entropy subadditivity gives $H(Z)\le\sum_i h(p_i)$. Combining the upper and lower bounds for $H(Z)$ gives
\begin{equation}\label{eq:oneA}
 \log_2T+\log_2|\cA|
 \le \left(n-\sum_i a_i\right)L
     +\sum_i\bigl[p_i h(a_i)+h(p_i)\bigr].
\end{equation}
Interchanging $\cA$ and $\cB$ gives
\begin{equation}\label{eq:oneB}
 \log_2T+\log_2|\cB|
 \le \left(n-\sum_i b_i\right)L
     +\sum_i\bigl[q_i h(b_i)+h(q_i)\bigr].
\end{equation}
Adding \eqref{eq:oneA} and \eqref{eq:oneB}, the left-hand side is
\[
 2\log_2T+\log_2|\cA|+\log_2|\cB|=3\log_2T.
\]
The right-hand side is
\[
 \sum_i\Bigl[
 p_i h(a_i)+q_i h(b_i)+h(p_i)+h(q_i)+L(2-a_i-b_i)
 \Bigr].
\]
Since $p_i=a_i(1-b_i)$ and $q_i=b_i(1-a_i)$, the bracket is
$F(a_i,b_i)+L(2-a_i-b_i)$ by \eqref{eq:F}. This proves \eqref{eq:direct}.
\end{proof}

\section{Proof of the main theorem}\label{sec:proof}
We combine the estimates of Section~\ref{sec:bounds}. Subsection~\ref{subsec:marginals} shows that a counterexample with minimum ground-set size has restricted coordinate probabilities. Subsection~\ref{subsec:combine} combines the bounds and completes the proof.

Set
\begin{equation}\label{eq:mu}
 \mu=\frac{223949}{100000},\qquad
 Q(x,y)=\frac{F(x,y)+(2-x-y)L}{3},
\end{equation}
and, for every $\rho>1$, define
\begin{equation}\label{eq:domain}
 \mathcal M_\rho=
 \bigl\{(x,y)\in[0,1]^2:
 \max\{x,1-x\}\max\{y,1-y\}\le\rho^{-1}\bigr\}.
\end{equation}
Lemma~\ref{lem:direct} and Theorem~\ref{thm:W} give
\begin{equation}\label{eq:global}
 \log_2T\le\sum_iQ(a_i,b_i),\qquad
 \log_2T\le\sum_iW_{\lambda,\tau}(a_i,b_i)
 \quad(0\le\lambda\le\tau\le1).
\end{equation}

\subsection{Bounds on the coordinate marginals}\label{subsec:marginals}
\begin{lemma}\label{lem:domain}
Let $\rho>1$. Suppose that $(\cA,\cB)$ is a recovering pair of nonempty families on $[n]$ with
$T=|\cA||\cB|>\rho^n$, and suppose that $n$ is minimal among all such counterexamples to the bound $T\le\rho^n$. Let $A$ and $B$ be independent and uniform on $\cA$ and $\cB$, and put
$a_i=\Prob(i\in A)$ and $b_i=\Prob(i\in B)$. Then, for every $i\in[n]$,
\begin{equation}\label{eq:domain-strict}
 \max\{a_i,1-a_i\}\max\{b_i,1-b_i\}<\rho^{-1}.
\end{equation}
In particular, $(a_i,b_i)\in\mathcal M_\rho$ for every $i$.
\end{lemma}

\begin{proof}
The minimal $n$ is at least one, since on the empty ground set both nonempty families are $\{\varnothing\}$. Fix $i\in[n]$ and $s,t\in\{0,1\}$. Restrict $\cA$ to sets with indicator $A_i=s$ and $\cB$ to sets with indicator $B_i=t$, and then delete coordinate $i$. Deletion is injective on each restricted family because the deleted bit is fixed. Separately, the $i$-th coordinates of both differences $A\setminus B$ and $B\setminus A$ are determined by the fixed pair $(s,t)$. Hence equality of either difference after deletion implies equality of the same difference before deletion. Applying the recovering property before deletion shows that every nonempty restricted pair is recovering.

By minimality, the product of the two restricted family sizes is at most $\rho^{n-1}$; if one restricted family is empty, the product is zero. Because $A$ and $B$ are uniform, the restricted family sizes are
\[
 |\cA|\,\Prob(A_i=s)
 \qquad\text{and}\qquad
 |\cB|\,\Prob(B_i=t),
\]
respectively. Their product is therefore
$T\Prob(A_i=s)\Prob(B_i=t)$, and hence
\[
 T\Prob(A_i=s)\Prob(B_i=t)\le\rho^{n-1}<T/\rho.
\]
Divide by $T$ and maximize over $s,t$. The left-hand side becomes
$\max\{a_i,1-a_i\}\max\{b_i,1-b_i\}$, proving \eqref{eq:domain-strict}.
\end{proof}

\subsection{Combining the bounds}\label{subsec:combine}
\begin{proposition}\label{prop:criterion}
Let $\rho>1$. Let $\omega_0,\ldots,\omega_s\ge0$ satisfy
$\sum_{j=0}^s\omega_j=1$, and let
$0\le\lambda_j\le\tau_j\le1$ for $1\le j\le s$. Suppose that
\begin{equation}\label{eq:criterion}
 \omega_0Q(x,y)+\sum_{j=1}^s\omega_jW_{\lambda_j,\tau_j}(x,y)
 \le\log_2\rho
 \qquad\text{for every }(x,y)\in\mathcal M_\rho.
\end{equation}
Then every recovering pair $(\cA,\cB)$ of nonempty families on an $n$-element set satisfies
$|\cA||\cB|\le\rho^n$.
\end{proposition}

\begin{proof}
Assume otherwise and choose a counterexample with minimum ground-set size $n$. With the notation of \eqref{eq:basic-notation}, Lemma~\ref{lem:domain} gives
$(a_i,b_i)\in\mathcal M_\rho$ for every $i$. Multiply the first inequality in \eqref{eq:global} by $\omega_0$ and, for each $1\le j\le s$, multiply the second inequality with parameters $(\lambda_j,\tau_j)$ by $\omega_j$. Because the weights sum to one, the weighted sum of the left-hand sides is still $\log_2T$. Hence
\[
 \log_2T
 \le\sum_{i=1}^n
 \left[\omega_0Q(a_i,b_i)+
 \sum_{j=1}^s\omega_jW_{\lambda_j,\tau_j}(a_i,b_i)\right].
\]
Apply \eqref{eq:criterion} to each coordinate. Then
$\log_2T\le n\log_2\rho$, contradicting $T>\rho^n$.
\end{proof}

For $\rho=\mu$ in \eqref{eq:mu}, use the following rational parameters:
\begin{center}\small
\begin{tabular}{@{}cccc@{}}\toprule
$j$ & $\lambda_j$ & $\tau_j$ & $\omega_j$\\\midrule
$0$ & --- & --- & $62269/100000$\\
$1$ & $26789/100000$ & $33329/100000$ & $31/4000$\\
$2$ & $26789/100000$ & $34121/100000$ & $269/4000$\\
$3$ & $27593/100000$ & $38331/100000$ & $2011/20000$\\
$4$ & $7/25$ & $35649/100000$ & $6931/50000$\\
$5$ & $26587/100000$ & $7457/20000$ & $3157/50000$\\\bottomrule
\end{tabular}
\end{center}
They satisfy $\omega_j>0$, $\sum_{j=0}^5\omega_j=1$, and
$0\le\lambda_j\le\tau_j\le1$. Define
\begin{equation}\label{eq:G}
 G(x,y)=\omega_0Q(x,y)+\sum_{j=1}^5\omega_jW_{\lambda_j,\tau_j}(x,y).
\end{equation}

\begin{proposition}\label{prop:pointwise}
For these parameters,
\begin{equation}\label{eq:certified}
 G(x,y)\le\log_2\mu
 \qquad\text{for every }(x,y)\in\mathcal M_\mu.
\end{equation}
\end{proposition}
The proof of Proposition~\ref{prop:pointwise} is given in Appendix~\ref{app:computer}.

\begin{proof}[Proof of Theorem~\ref{thm:main}]
Apply Proposition~\ref{prop:criterion} with $\rho=\mu$, the six weights and the five parameter pairs specified above. The table gives positive weights summing to one and parameters satisfying $0\le\lambda_j\le\tau_j\le1$, and Proposition~\ref{prop:pointwise} gives \eqref{eq:criterion}. Therefore
$|\cA||\cB|\le\mu^n=(223949/100000)^n$.
\end{proof}

\begin{remark}\label{rem:limitation}
Proposition~\ref{prop:criterion} uses finite convex combinations of the inequalities in \eqref{eq:global}. Proposition~\ref{prop:obstruction} shows that, when $\rho=2.2394$, no such finite convex combination can satisfy \eqref{eq:criterion} on $\mathcal M_\rho$. Thus Proposition~\ref{prop:criterion} together with only the inequalities in \eqref{eq:global} cannot prove the bound $|\cA||\cB|\le 2.2394^n$. This does not rule out arguments using other consequences of the recovering property.
\end{remark}

\appendix
\section{Computer-assisted inequalities}\label{app:computer}

This appendix contains the mathematical reductions for Proposition~\ref{prop:pointwise} and Proposition~\ref{prop:obstruction}. The computational supplement describes the finite rational checks that complete these proofs and identifies the corresponding ancillary material.

\begin{proof}[Proof of Proposition~\ref{prop:pointwise}]
For $1\le j\le5$, put
\[
 d_j=1+2\tau_j-\lambda_j,
 \qquad \Delta_j=\tau_j-\lambda_j.
\]
Write $p=x(1-y)$ and $q=y(1-x)$. From \eqref{eq:F} and \eqref{eq:mu},
\[
 Q(x,y)=\frac{p}{3}h(x)+\frac{q}{3}h(y)
        +\frac13h(p)+\frac13h(q)+\frac{2-x-y}{3}L.
\]
For the $j$-th function $W_{\lambda_j,\tau_j}$, substitute
\[
 g(x,y)=h(p)-(1-y)h(x),
 \qquad
 g(y,x)=h(q)-(1-x)h(y)
\]
into \eqref{eq:Vhat}, and then substitute the result into \eqref{eq:W}. After division by
$d_j=1+2\tau_j-\lambda_j$, the coefficient of $h(p)$ is
\[
 \frac{1-(x-\lambda_j)_++\Delta_j}{d_j},
\]
and the coefficient of $h(q)$ is obtained by interchanging $x$ and $y$. The coefficient of $h(x)$ is
\[
 \frac{(\lambda_j-y)_+ +(x-\lambda_j)_+(1-y)
 +(1-y)(\tau_j-\max\{\lambda_j,y\})_+}{d_j},
\]
with the coefficient of $h(y)$ again obtained by interchanging $x$ and $y$.
Multiplying these coefficients by the weights in \eqref{eq:G} and adding the contribution of $Q$ gives
\begin{align}\label{eq:coeff-appendix}
 G(x,y)={}&C_xh(x)+C_yh(y)+C_ph(p)+C_qh(q)
 +\frac{\omega_0}{3}(2-x-y)L,\notag\\
 C_x={}&\frac{\omega_0p}{3}+\sum_{j=1}^5\frac{\omega_j}{d_j}
 \Bigl[(\lambda_j-y)_+ +(x-\lambda_j)_+(1-y)\notag\\
 &\hspace{29mm}+(1-y)(\tau_j-\max\{\lambda_j,y\})_+\Bigr],\notag\\
 C_y={}&\frac{\omega_0q}{3}+\sum_{j=1}^5\frac{\omega_j}{d_j}
 \Bigl[(\lambda_j-x)_+ +(y-\lambda_j)_+(1-x)\notag\\
 &\hspace{29mm}+(1-x)(\tau_j-\max\{\lambda_j,x\})_+\Bigr],\notag\\
 C_p={}&\frac{\omega_0}{3}+\sum_{j=1}^5\frac{\omega_j}{d_j}
 \bigl[1-(x-\lambda_j)_++\Delta_j\bigr],\notag\\
 C_q={}&\frac{\omega_0}{3}+\sum_{j=1}^5\frac{\omega_j}{d_j}
 \bigl[1-(y-\lambda_j)_++\Delta_j\bigr].
\end{align}
Every coefficient in \eqref{eq:coeff-appendix} is nonnegative on $[0,1]^2$. The terms defining $C_x$ and $C_y$ are individually nonnegative. For $C_p$ and $C_q$, use
$0\le(x-\lambda_j)_+,(y-\lambda_j)_+\le1$ and $\Delta_j\ge0$.

For $t\in(0,1)$,
\[
 h'(t)=\log_2\frac{1-t}{t},
 \qquad
 h(t)-t h'(t)=-\log_2(1-t).
\]
Since $h$ is concave, its tangent at $t$ is a global upper bound. Thus, for every rational $t\in(0,1)$,
\begin{equation}\label{eq:tangent-appendix}
 h(z)\le-\log_2(1-t)+z\log_2\frac{1-t}{t}
 \qquad(0\le z\le1).
\end{equation}
Let $R=[a,b]\times[c,e]$ be a nondegenerate rational rectangle arising in the subdivision. Its midpoint has rational coordinates $x_0=(a+b)/2$ and $y_0=(c+e)/2$, with $0<x_0,y_0<1$. Hence the four tangent points
\[
 t_x=x_0,\qquad t_y=y_0,\qquad
 t_p=x_0(1-y_0),\qquad t_q=y_0(1-x_0)
\]
are rational numbers in $(0,1)$. For any such tangent point $t$, let
\[
 \underline{\ell}(t)\le \log_2(1-t),
 \qquad
 \overline{s}(t)\ge \log_2\frac{1-t}{t}
\]
be rational bounds used in the finite verification, and define the rational affine function
\[
 \overline h_t(z)=-\underline{\ell}(t)+\overline{s}(t)z.
\]
By \eqref{eq:tangent-appendix}, $h(z)\le\overline h_t(z)$ for $0\le z\le1$: the intercept is increased because $-\underline{\ell}(t)\ge-\log_2(1-t)$, and the slope is increased while $z\ge0$. The coefficients in \eqref{eq:coeff-appendix} are nonnegative, so replacing $h(x),h(y),h(p),h(q)$ by the corresponding rational affine majorants preserves an upper bound for $G$. We also replace $L=\log_2(9/4)$ by a rational upper bound and $\log_2\mu$ by a rational lower bound. These replacements preserve an upper bound for $G-\log_2\mu$ because $\omega_0(2-x-y)/3\ge0$ on $[0,1]^2$, while the lower bound for $\log_2\mu$ is subtracted. Denote by $P_R(x,y)$ the resulting rational polynomial upper bound for $G(x,y)-\log_2\mu$.

The initial partition of each coordinate uses $0$, $1/2$, $1$, and all distinct values among the $\lambda_j$ and $\tau_j$. On each resulting rectangle, and on every rectangle obtained from it by bisection, every positive-part expression in \eqref{eq:coeff-appendix} agrees with a fixed affine function. Since $p=x(1-y)$ and $q=y(1-x)$ are bilinear, the upper bound obtained after replacing the entropy terms by their affine bounds is a polynomial of bidegree at most $(2,2)$.

To decide whether a rectangle can intersect $\mathcal M_\mu$, note that the initial partition contains $1/2$ in both coordinates and every later rectangle is obtained by bisection. Hence each rectangle $R=[a,b]\times[c,e]$ lies in one of the four regions determined by $x\le1/2$ or $x\ge1/2$ and $y\le1/2$ or $y\ge1/2$. On these four regions the minimum of
$\max\{x,1-x\}\max\{y,1-y\}$ on $R$ is, respectively,
\begin{equation}\label{eq:exclude-appendix}
 (1-b)(1-e),\qquad (1-b)c,\qquad a(1-e),\qquad ac.
\end{equation}
Thus $R$ is disjoint from $\mathcal M_\mu$ whenever the corresponding quantity in \eqref{eq:exclude-appendix} is strictly larger than $1/\mu$; equality is not sufficient for exclusion.

The finite verification described in the computational supplement checks by rational arithmetic that the finite family of rectangles covers $[0,1]^2$ and that every rectangle $R$ in that family satisfies one of the following two alternatives:
\[
 R\cap\mathcal M_\mu=\varnothing,
 \qquad\text{or}\qquad
 P_R(x,y)<0\quad\text{for every }(x,y)\in R.
\]
The first alternative is checked using \eqref{eq:exclude-appendix}. If $(x,y)\in\mathcal M_\mu$, the rectangle containing $(x,y)$ must therefore satisfy the second alternative, and by the definition of $P_R$,
\[
 G(x,y)-\log_2\mu\le P_R(x,y)<0.
\]
This proves Proposition~\ref{prop:pointwise}.
\end{proof}

Lemma~\ref{lem:param-reduction} reduces the minimization in Proposition~\ref{prop:obstruction} to finitely many parameter pairs.

\begin{lemma}\label{lem:param-reduction}
Let $\nu$ be a finitely supported probability measure on $[0,1]^2$, and put
\[
 S_\nu=\{0,1\}\cup\{x,y:(x,y)\in\operatorname{supp}\nu\}.
\]
Then the minimum of
$\int W_{\lambda,\tau}\,d\nu$ over $0\le\lambda\le\tau\le1$ is attained at a pair
$(\lambda,\tau)$ with $\lambda,\tau\in S_\nu$.
\end{lemma}

\begin{proof}
For $0\le\lambda\le\tau\le1$ and every $x\in[0,1]$,
\[
 (\tau-\max\{\lambda,x\})_+
 =(\tau-x)_+-(\lambda-x)_+.
\]
After this substitution in \eqref{eq:W}, the numerator of
$\int W_{\lambda,\tau}\,d\nu$ is affine in $(\lambda,\tau)$ on each cell cut out by the lines
$\lambda=z$, $\tau=z$ for $z\in S_\nu$, together with $\lambda=\tau$. The denominator
$1+2\tau-\lambda$ is positive throughout the parameter triangle. Parameterize a line segment contained in one cell by $s\in[0,1]$. Along the segment the quotient has the form
\[
 \frac{\alpha s+\beta}{\gamma s+\delta},
 \qquad \gamma s+\delta>0.
\]
Its derivative is
$ (\alpha\delta-\beta\gamma)/(\gamma s+\delta)^2$, so the quotient is monotone or constant on the segment. Given a point of the polygonal cell, choose a line through it whose intersection with the cell is a segment with endpoints on the boundary. Monotonicity on that segment shows that one boundary endpoint has value no larger than the original point. Each boundary endpoint lies on an edge, and monotonicity along that edge shows that one of its vertices has value no larger. Hence the minimum over the cell is attained at a vertex. Every vertex has both coordinates in $S_\nu$; this includes vertices on the diagonal because $0,1\in S_\nu$. Taking the minimum over the finitely many cells proves the claim.
\end{proof}

\begin{proposition}\label{prop:obstruction}
Put $\mu_-=11197/5000=2.2394$ and $m_-=\log_2\mu_-$. There exists a symmetric rational probability measure $\nu$ supported strictly inside $\mathcal M_{\mu_-}$ such that
\begin{equation}\label{eq:obstruction}
 \int Q\,d\nu>m_-+10^{-5},\qquad
 \int W_{\lambda,\tau}\,d\nu>m_-+10^{-5}
 \quad\text{for every }0\le\lambda\le\tau\le1.
\end{equation}
Consequently no finite convex combination of $Q$ and the functions
$W_{\lambda,\tau}$ can be bounded above by $m_-$ on all of $\mathcal M_{\mu_-}$.
\end{proposition}

\begin{proof}
Let $\delta_z$ denote unit mass at $z$ and define
\[
 \nu=\frac12\sum_{i=1}^4v_i
 \bigl(\delta_{(x_i,y_i)}+\delta_{(y_i,x_i)}\bigr),
\]
where
\begin{center}\small
\begin{tabular}{@{}cccc@{}}\toprule
$i$&$x_i$&$y_i$&$v_i$\\\midrule
$1$&$2619177/10000000$&$3949887/10000000$&$142569/500000$\\
$2$&$330333/1250000$&$3930573/10000000$&$1499169/10000000$\\
$3$&$33003/100000$&$833701/2500000$&$2086751/10000000$\\
$4$&$57/200$&$3932489/10000000$&$35627/100000$\\\bottomrule
\end{tabular}
\end{center}
The weights sum to one. Both $Q$ and $W_{\lambda,\tau}$ are symmetric in their two arguments: this follows directly from \eqref{eq:F}, \eqref{eq:mu}, \eqref{eq:Vhat}, and \eqref{eq:W} by interchanging $x$ and $y$. The finite verification described in the computational supplement uses all eight support points in the table, each with mass $v_i/2$, and checks by rational comparison that every support point lies strictly inside $\mathcal M_{\mu_-}$.

For this measure, $|S_\nu|=10$, so Lemma~\ref{lem:param-reduction} reduces the second inequality in \eqref{eq:obstruction} to the $55$ ordered pairs $\lambda,\tau\in S_\nu$ with $\lambda\le\tau$. The finite verification also establishes by rational arithmetic that the first inequality in \eqref{eq:obstruction} and the second inequality for each of these $55$ pairs hold with the margin stated in \eqref{eq:obstruction}. Lemma~\ref{lem:param-reduction} then gives the second inequality for every $0\le\lambda\le\tau\le1$.

Suppose that a finite convex combination
\[
 R(x,y)=\alpha_0Q(x,y)+\sum_{j=1}^r\alpha_jW_{\lambda_j,\tau_j}(x,y),
 \qquad \alpha_j\ge0,\quad \sum_{j=0}^r\alpha_j=1,
\]
satisfied $R(x,y)\le m_-$ for every $(x,y)\in\mathcal M_{\mu_-}$. Since $\nu$ is supported in $\mathcal M_{\mu_-}$, integration would give $\int R\,d\nu\le m_-$. On the other hand, \eqref{eq:obstruction} gives
\[
 \int R\,d\nu
 =\alpha_0\int Q\,d\nu
  +\sum_{j=1}^r\alpha_j\int W_{\lambda_j,\tau_j}\,d\nu
 >m_-+10^{-5},
\]
a contradiction.
\end{proof}

\section*{Acknowledgment of AI assistance}
The overall research direction and the strategy of combining complementary estimates were formulated by the author. ChatGPT (OpenAI; GPT-6 Astra and GPT-5.6 Sol) was used to produce initial versions of all proofs retained in the paper and all code used in the computer-assisted parts, as well as to perform exploratory calculations and assist with drafting. The author independently verified every mathematical argument and every computer-assisted claim retained in the final manuscript, revised the exposition, and assumes full responsibility for the correctness and content of the paper.


\begin{thebibliography}{99}
\bibitem{AS} R.~Ahlswede and G.~Simonyi, \emph{On the optimal structure of recovering set pairs in lattices: the sandglass conjecture}, Discrete Math. \textbf{128} (1994), 389--394. doi:10.1016/0012-365X(94)90129-5.
\bibitem{Chang} F.~Chang, \emph{A hypercontractive proof of the sharp bound for cancellative pairs}, arXiv:2609.09105 (2026).
\bibitem{DFCDhashing} S.~Della Fiore, S.~Costa and M.~Dalai, \emph{Improved bounds for $(b,k)$-hashing}, IEEE Trans. Inform. Theory \textbf{68} (2022), no.~8, 4983--4997. doi:10.1109/TIT.2022.3167608.
\bibitem{FH} Y.~Fang and H.~Huang, \emph{On cancellative pairs of families of subsets}, arXiv:2609.01483v2 (2026).
\bibitem{HK} R.~Holzman and J.~K\"orner, \emph{Cancellative pairs of families of sets}, European J. Combin. \textbf{16} (1995), 263--266. doi:10.1016/0195-6698(95)90031-4.
\bibitem{Janzer} B.~Janzer, \emph{A new upper bound for cancellative pairs}, Electron. J. Combin. \textbf{25} (2018), \#P2.13. doi:10.37236/7210.
\bibitem{MSV} A.~Mond, V.~Souza and L.~Versteegen, \emph{The sandglass conjecture beyond cancellative pairs}, arXiv:2508.21819v1 (2025).
\bibitem{NY} C.~Nair and M.~Yazdanpanah, \emph{On the AND--OR interference channel and the sandglass conjecture}, in: 2020 IEEE International Symposium on Information Theory (ISIT), 1540--1545. doi:10.1109/ISIT44484.2020.9174073.
\bibitem{Soltesz} D.~Solt\'esz, \emph{New bounds on Simonyi's conjecture}, European J. Combin. \textbf{70} (2018), 251--267. doi:10.1016/j.ejc.2018.01.005.
\bibitem{Tolhuizen} L.~M.~Tolhuizen, \emph{New rate pairs in the zero-error capacity region of the binary multiplying channel without feedback}, IEEE Trans. Inform. Theory \textbf{46} (2000), 1043--1046. doi:10.1109/18.841182.
\end{thebibliography}
\end{document}